\documentclass[11pt]{article}
\usepackage{multirow}
\usepackage{mathrsfs}
\usepackage[title]{appendix}
\usepackage{manyfoot}
\usepackage{booktabs}
\usepackage{float}
\usepackage[ruled]{algorithm}
\usepackage{algpseudocode}
\algrenewcommand\algorithmicrequire{\textbf{Require:}}
\algrenewcommand\algorithmicensure{\textbf{Ensure:}}
\usepackage{listings}
\usepackage{enumitem}
\usepackage{makecell}
\usepackage{placeins}
\usepackage{graphicx}
\usepackage{xcolor}
\usepackage{adjustbox}

\usepackage{amsmath,amssymb,amsfonts,amsthm,mathtools}
\pdfmapfile{+/usr/share/texlive/texmf-dist/fonts/map/dvips/amsfonts/cm.map}
\pdfmapfile{+/usr/share/texlive/texmf-dist/fonts/map/dvips/amsfonts/cmextra.map}
\usepackage{txfonts}
\usepackage{authblk}

\usepackage{geometry}
\usepackage{array}

\usepackage[colorlinks=true]{hyperref}

\usepackage{cleveref}

\newcommand{\doilink}[1]{%
	\href{https://doi.org/#1}{https://doi.org/#1}%
}

\hypersetup{colorlinks=true,linkcolor=blue,citecolor=blue,urlcolor=blue}

\makeatletter
\providecommand*{\toclevel@algorithm}{0}
\providecommand*{\theHALG@line}{}
\renewcommand*{\theHALG@line}{\thealgorithm.\arabic{ALG@line}}
\makeatother

\newcolumntype{L}[1]{>{\raggedright\arraybackslash}p{#1}}

\newtheorem{theorem}{Theorem}[section]
\newtheorem{lemma}[theorem]{Lemma}
\newtheorem{proposition}[theorem]{Proposition}
\newtheorem{corollary}[theorem]{Corollary}
\newtheorem{assumption}[theorem]{Assumption}
\newtheorem{definition}[theorem]{Definition}

\theoremstyle{remark}
\newtheorem{remark}[theorem]{Remark}

\newcommand{\R}{\mathbb{R}}
\newcommand{\norm}[1]{\left\|#1\right\|}
\newcommand{\grad}{\nabla}
\newcommand{\hess}{\nabla^2}
\DeclareMathOperator*{\argmin}{arg\,min}
\DeclareMathOperator{\co}{co}
\DeclareMathOperator{\diag}{diag}
\newcommand{\clarke}{\partial}
\newcommand{\ytilde}{\widetilde y}
\newcommand{\Cmode}{\mathsf C}
\newcommand{\Ncorr}{\mathrm{N}}
\newcommand{\Ecorr}{\mathrm{E}}
\newcommand{\Hcorr}{\mathrm{H}}
\newcommand{\THMLBFGSN}{\textnormal{\textsc{HMLBFGS-N}}}
\newcommand{\THMLBFGSE}{\textnormal{\textsc{HMLBFGS-E}}}
\newcommand{\THMLBFGSH}{\textnormal{\textsc{HMLBFGS-H}}}

\newcommand{\Iset}{\mathcal I}
\newcommand{\deltac}{m_{\rm cur}}
\newcommand{\Lc}{M_{\rm cur}}

\DeclareMathOperator{\clip}{clip}

\title{
	Preconditioned Hyperbolic Smoothing Modified L-BFGS Algorithms
	for Finite Minimax Problems
}

\author[1]{Wenzhe Zhao}

\affil[1]{National Center for Applied Mathematics in Chongqing,
	Chongqing Normal University, Chongqing 401331, China}

\date{}

\date{}

\begin{document}
	
	\maketitle
	
	% ============================================================
	% Abstract
	% ============================================================
	
	\begin{abstract}
		This paper proposes a preconditioned ordinary hyperbolic smoothing modified L-BFGS framework for finite minimax problems.
		To mitigate the deterioration of Euclidean conditioning as the smoothing parameter decreases, a problem-derived hyperbolic majorization preconditioner is incorporated into the initial L-BFGS metric; it combines component-gradient Lipschitz curvature with the curvature induced by the hyperbolic smoothing Jacobian, is uniformly positive definite, and provably majorizes the Hessian of the smoothed objective.
		To exploit strict negative secant curvature, three correction strategies are introduced---direct sign reflection, a Euclidean nearest-point correction, and a $B_k^{-1}$-metric nearest-point correction---for which explicit formulas are derived and Clarke-stationary accumulation points are obtained for the idealized continuation.
		For a fixed smoothing parameter, Q-linear convergence of the HMLBFGS objective values is obtained under a Polyak--Lojasiewicz condition, local strong convexity further yields R-linear convergence of the iterates, and the corresponding full-memory modified BFGS methods attain local Q-superlinear convergence under the usual unit-step assumptions.
		Numerical experiments demonstrate the effectiveness of the proposed preconditioned methods and their numerical advantages over the selected comparison methods.
	\end{abstract}
	
	\noindent
	\textbf{Keywords:}
	finite minimax problem; nonconvex large-scale problems; hyperbolic smoothing; BFGS.
	
	% ============================================================
	\section{Introduction}
	\label{sec:introduction}
	% ============================================================
	Consider the unconstrained finite minimax problem
	\begin{equation}
		\label{eq:minimax}
		\min_{x\in\R^n} f(x),
		\qquad
		f(x):=\max_{i\in\Iset} f_i(x),
		\qquad
		\Iset:=\{1,\ldots,m\}.
	\end{equation}
	From a theoretical perspective, finite minimax problems are closely related
	to nonlinear programming~\cite{Jiang2021}, systems of nonlinear equations~\cite{Liuhaoyang2020}, linear complementarity problems~\cite{Wang2008}, and multiobjective optimization~\cite{Nogueira2026}. They also arise in a broad range of applications, including data fitting~\cite{BertsimasBrownCaramanis2011}, structural optimization under
	load uncertainty~\cite{NishiokaKanno2023}, location problems~\cite{DreznerZerom2023}, resource allocation~\cite{Park2021}, and group-robust learning~\cite{Pillutla2024}. 
	
	Existing methods for finite minimax problems can be broadly classified into two categories. The first category consists of nonsmooth methods, which directly solve the original problem, including subgradient methods~\cite{ChatelonHearnLowe1978}, gradient-sampling methods~\cite{BurkeLewisOverton2005}, cutting plane methods~\cite{GaudiosoEtAl2022} and bundle methods~\cite{TangJianLi2019,DiazGrimmer2023}. These methods typically construct search directions based on negative subgradients or suitable combinations of subgradients, and then generate the next iterate via a line-search procedure. The second category consists of smoothing methods, which replace the objective function $f$ with a family of smooth approximations, thereby allowing optimization techniques developed for smooth problems to be employed. Smoothing methods can be further classified into local smoothing methods and global smoothing methods. Local smoothing methods \cite{Zang1980,BagirovTaheri2018} regularizes the objective only near kink points, namely, points at which the objective function is
	nondifferentiable. However, local smoothing may generate spurious local minimizers through combinations of	adjacent shallow local minimizers. To avoid this drawback, global smoothing
	methods, which regularize the objective throughout its domain, have received considerable attention.
	
	Global smoothing methods typically employ exponential smoothing or hyperbolic smoothing functions.
	Aggregate exponential smoothing is a widely used global smoothing technique
	for finite minimax problems~\cite{Li1992,PolakRoysetWomersley2003,XiaoYu2010,LiuZheng2020}. However, exponential smoothing functions may suffer from numerical overflow, particularly when the number of component functions is large.
	Hyperbolic smoothing avoids exponential evaluations and hence avoids
	this overflow issue~\cite{BagirovTaheri2018,Bagirov2013}. Here we use the ordinary full model and do not discard any component.
	
	One of the principal advantages of smoothing approaches is that they allow one to exploit a broad class of well-established optimization algorithms for smooth problems, including steepest descent methods~\cite{LiuZheng2020}, accelerated gradient methods, stochastic gradient methods, Newton methods~\cite{YeLiuZhouLiu2008}, quasi-Newton methods~\cite{LiuNocedal1989,NocedalWright2006}, trust-region methods, and conjugate gradient methods~\cite{PangDuJu2016,GuoWan2022}. Among quasi-Newton methods,
	the BFGS method is widely used because of its favorable convergence behavior
	and modest computational cost~\cite{DennisMore1977}; its limited-memory variant, L-BFGS, is
	particularly suitable for large-scale problems~\cite{LiuNocedal1989,NocedalWright2006}. Classical BFGS analyses commonly use the positive
	curvature condition implied by a Wolfe-type line search. In contrast, an
	Armijo step on a nonconvex smoothed minimax model can yield $s_k^Ty_k<0$.
	To overcome this difficulty, existing methods typically employ strategies such as skipping the update, using damped updates~\cite{Powell1978}, introducing regularization terms~\cite{Mannel2025,KanzowSteck2023}, or relaxing the secant condition~\cite{BerglundZhangJohansson2025}.
	For example, Li and Fukushima
	~\cite{LiFukushima2001a,LiFukushima2001b} proposed modified BFGS updates for
	nonconvex minimization; Powell damping~\cite{Powell1978} blends the observed
	curvature vector with a positive-curvature vector; cautious and regularized
	L-BFGS updates~\cite{Mannel2025,KanzowSteck2023} control the use of
	curvature pairs; and soft quasi-Newton updating~\cite{BerglundZhangJohansson2025}
	relaxes the secant constraint. These methods preserve a positive definite
	metric, but they may discard a negative-curvature pair or replace part of its
	observed curvature information in order to recover a positive-definite model.
	
	The starting point of our design is the following question: how can one
	mitigate the increasing ill-conditioning induced by hyperbolic smoothing while,
	at the same time, retaining and exploiting negative-curvature information to
	construct effective descent directions?  These two issues become particularly
	important in the later stages of smoothing continuation, where a small
	smoothing parameter is required to approximate the original finite minimax
	problem accurately.
	
	The source of the first difficulty can be seen directly from the structure of
	the Hessian of the smoothed objective.  With the notation introduced below,
	\begin{equation}
		\label{eq:full-hessian-intro}
		\nabla^2\Phi_\tau(x,t)
		=
		\frac12\sum_{i\in\Iset}
		\left[
		\bigl(1+\beta_{i,\tau}(x,t)\bigr)D_i(x)
		+
		\vartheta_{i,\tau}(x,t)q_i(x)q_i(x)^T
		\right],
	\end{equation}
	where
	\[
	\vartheta_{i,\tau}(x,t)
	=
	\frac{\tau^2}
	{\bigl((f_i(x)-t)^2+\tau^2\bigr)^{3/2}}.
	\]
	The first term in \eqref{eq:full-hessian-intro} inherits the curvature of the
	component functions, whereas the second rank-one term is generated by the
	hyperbolic smoothing.  Near a switching surface \(f_i(x)=t\), one has
	\(\vartheta_{i,\tau}=1/\tau\); hence this smoothing-induced curvature becomes
	increasingly large as \(\tau\downarrow0\), while curvature in other directions
	need not grow at the same rate.  More precisely, on a bounded set
	\(\mathcal Q\), the gradient-Lipschitz bound satisfies
	\begin{equation}
		\label{eq:tau-dependent-L-intro}
		L_{\mathcal Q,\tau}
		=
		m\omega_{\mathcal Q}
		+
		\frac{m}{2\tau}
		\bigl(\gamma_{\mathcal Q}^2+1\bigr),
	\end{equation}
	which deteriorates at the rate \(O(\tau^{-1})\).  Thus the worsening
	conditioning is a structural consequence of the hyperbolic smoothing itself,
	rather than merely a numerical artifact.
	
	This observation motivates a preconditioner tailored to the same curvature
	structure.  Instead of using an ad hoc diagonal scaling, we construct
	\begin{equation}
		\label{eq:Pk-intro}
		P(z,\tau)
		=
		\delta I_{n+1}
		+
		\bar L\Pi_x
		+
		\frac12J(z)^TD^{\rm hyp}(z,\tau)J(z),
	\end{equation}
	where, with
	\(\omega_i=\bigl((f_i(x)-t)^2+\tau^2\bigr)^{1/2}\),
	\(D^{\rm hyp}(z,\tau)=\operatorname{diag}(\omega_1^{-1},\ldots,
	\omega_m^{-1})\).  The construction follows directly from the Hessian
	structure above.  Indeed,
	\[
	\vartheta_{i,\tau}
	=
	\frac{\tau^2}{\omega_i^3}
	\le \frac{1}{\omega_i},
	\]
	so the term \(J^TD^{\rm hyp}J/2\) majorizes the rank-one curvature generated
	by hyperbolic smoothing and reproduces its \(O(\tau^{-1})\) scaling near
	switching surfaces, while \(\bar L\Pi_x\) controls the curvature inherited
	from the component functions.  After adding \(\delta I_{n+1}\), the
	resulting matrix is uniformly positive definite and provably majorizes the
	Hessian of the smoothed objective.  We therefore use
	\(H_{k,0}=\chi_kP_k^{-1}\) as the initial inverse L-BFGS metric, so that the
	directions in which hyperbolic smoothing introduces excessive curvature are
	rescaled before the limited-memory secant information is accumulated.  Thus
	the proposed preconditioner is not a generic scaling device; it is constructed
	specifically from the mechanism that causes the smoothing problem to become
	ill-conditioned as \(\tau\downarrow0\).
	
	Ill-conditioning, however, is only one part of the difficulty.  Since the
	smoothed minimax model is generally nonconvex, an accepted Armijo step can
	produce a secant pair
	\[
	s_k=z_{k+1}-z_k,\qquad
	y_k=g_{k+1}-g_k,
	\]
	with \(s_k^Ty_k<0\).  As reviewed above, many quasi-Newton strategies restore
	positive definiteness by skipping such a pair, damping it, or adding a
	regularizing modification that overrides the negative curvature.  Our
	viewpoint is that the computed vector \(y_k\) does not cease to be
	informative when \(s_k^Ty_k<0\): it still contains curvature information from the underlying
	smoothed problem, and the negative secant product records negative average
	curvature along the accepted displacement.  Rather than discarding this
	information, we seek to transform it into an admissible positive-curvature
	pair from which a positive-definite quasi-Newton model, and hence an acceptable
	descent direction, can be constructed.
	
	A natural first choice is direct sign reflection,
	\[
	\widetilde y_k^{\Ncorr}=-y_k,
	\]
	for which
	\(s_k^T\widetilde y_k^{\Ncorr}=-s_k^Ty_k=|s_k^Ty_k|>0\).
	If, in addition to enforcing the same target curvature
	\[
	s_k^T\widetilde y_k=-s_k^Ty_k,
	\]
	we require the corrected vector to remain as close as possible to the
	observed curvature vector \(y_k\), then two further corrections arise
	naturally from the metric-projection problem
	\begin{equation}
		\label{eq:design-problem-intro}
		\min_v\;
		\frac12\|v-y_k\|_M^2
		\quad\text{subject to}\quad
		s_k^Tv=-s_k^Ty_k.
	\end{equation}
	Choosing \(M=I\) gives the Euclidean minimum-change correction, whereas
	choosing \(M=B_k^{-1}\) gives the minimum-change correction in the current
	quasi-Newton geometry.  Together with direct sign reflection, these two
	projections provide three distinct ways of converting an observed
	negative-curvature pair into an admissible positive-secant pair while
	retaining, in different geometries, as much of the original curvature
	information as possible.
	
	Combining the hyperbolic-structure-based preconditioner with these three
	negative-curvature corrections leads to a family of preconditioned
	hyperbolic smoothing modified L-BFGS methods for large-scale finite minimax
	problems.  The preconditioner is designed to counteract the
	smoothing-induced deterioration of conditioning, whereas the correction
	mechanism makes explicit use of otherwise discarded negative-curvature
	information.  The contributions are as follows.

	\begin{enumerate}[label=(\arabic*)]
		\item A problem-derived hyperbolic-majorization preconditioner is
		constructed to alleviate the deterioration of conditioning as
		\(\tau\downarrow0\).  It combines a bound on the component curvature with
		the curvature induced by the hyperbolic smoothing Jacobian, is uniformly
		positive definite, and provably majorizes the Hessian of the smoothed
		objective.  It is incorporated through the initial inverse metric
		\(H_{k,0}=\chi_kP_k^{-1}\), so the L-BFGS recursion starts in a geometry
		adapted to the dominant smoothing curvature.
		\item Three modified BFGS rules are developed to use strict negative
		curvature rather than discard it: direct sign reflection, a Euclidean
		nearest-point correction, and a \(B_k^{-1}\)-metric nearest-point
		correction.  Each rule gives the same prescribed positive secant
		curvature while changing the observed curvature information in a
		controlled manner.
		\item A preconditioned ordinary hyperbolic smoothing modified L-BFGS
		framework is developed for finite minimax problems.  The continuation,
		Armijo line search, preconditioned BB2-type scaling, limited-memory
		recursion, and memory safeguards are common to \THMLBFGSN, \THMLBFGSE,
		and \THMLBFGSH; only the negative-curvature pair correction changes.
		\item The resulting safeguards yield a fixed-layer uniformly bounded SPD
		L-BFGS metric and descent directions.  Under idealized continuation with
		vanishing inner tolerances, every accumulation point is Clarke stationary
		for the original minimax objective.  For a fixed smoothing parameter, the HMLBFGS
		objective values converge Q-linearly under a PL condition, while local
		strong convexity yields R-linear convergence of the iterates.  For the
		corresponding full-memory modified BFGS methods, local
		\(Q\)-superlinear convergence is established under local strong
		convexity.  Numerical experiments illustrate the effectiveness and
		numerical advantages of the proposed methods.
	\end{enumerate}
	
	The remainder of the paper is organized as follows.
	Section~\ref{sec:prelim} introduces notation and the finite-minimax
	stationarity model, and Section~\ref{sec:hyperbolic} describes the ordinary
	hyperbolic smoothing model.  Section~\ref{sec:preconditioner} develops the
	hyperbolic-majorization preconditioner used to address the conditioning
	deterioration of the smoothed problem.  The three negative-curvature BFGS
	corrections are analyzed in Section~\ref{sec:three-corrections}, followed by
	the preconditioned L-BFGS recursion in Section~\ref{sec:lbfgs}.
	Section~\ref{sec:algorithm} gives the complete HMLBFGS framework and the
	corresponding BFGS update.  Section~\ref{sec:global} establishes global
	convergence, and Section~\ref{sec:lbfgs-linear} gives linear convergence
	rates for the fixed-\(\tau\) HMLBFGS inner iterations.  The local behavior
	of the corresponding full-memory BFGS methods is studied in
	Section~\ref{sec:local}.  Section~\ref{sec:numerics} presents numerical
	experiments, and Section~\ref{sec:conclusion} concludes the paper.
	
	% ============================================================
	\section{Preliminaries}
	\label{sec:prelim}
	% ============================================================
	
	This section fixes the notation used by the smoothing model and the convergence analysis. In~\eqref{eq:minimax},
	$\Iset:=\{1,\ldots,m\}$, $f_i:\R^n\to\R$, and
	$f(x):=\underset{i\in\Iset}{\max}f_i(x)$.
	The lifted variable is $z=(x,t)\in\R^{n+1}$ and $0_n$ is the zero vector
	in $\R^n$. The inner product and norm are
	the Euclidean ones, denoted by $\langle\cdot,\cdot\rangle$ and
	$\norm{\cdot}$, respectively. When applied to a matrix,
	$\norm{\cdot}$ denotes the induced Euclidean norm. For symmetric matrices,
	$\mathbf A\succ0$ and $\mathbf A\succeq0$ mean positive definite and positive semidefinite.
	
	For a set $B$, $\co(B)$ and $|B|$ denote its convex hull and cardinality.
	The active index set and lower level set are
	\begin{align}
		\label{eq:basic-sets}
		\Iset(x)&:=\{i\in\Iset:f_i(x)=f(x)\},&
		\mathcal L_{f,c}&:=\{x\in\R^n:f(x)\le c\},\\
		\gamma(x)&:=\max_{i\in\Iset}\norm{\grad f_i(x)},&
		\omega(x)&:=\max_{i\in\Iset}\norm{\hess f_i(x)}.
		\label{eq:gamma-omega}
	\end{align}
	Since a finite maximum of continuously differentiable functions is locally
	Lipschitz, its Clarke subdifferential is given by the standard finite-max
	rule
	\begin{equation}
		\label{eq:clarke-max}
		\clarke f(x)
		=\co\{\grad f_i(x):i\in\Iset(x)\}.
	\end{equation}
	Throughout the paper, $\partial$ denotes the Clarke subdifferential whenever it is applied to a locally Lipschitz function. A point $x$ is called Clarke stationary if $0\in\partial f(x)$.
	
	If $\psi:\R^d\to\R$ be differentiable and let
	\[
	\psi^*:=\inf_{u\in\R^d}\psi(u)>-\infty.
	\]
	\begin{definition}[Polyak--Lojasiewicz condition,~\cite{AbbaszadehpeivastiEtAl2023}]
		\label{def:pl}
		For a set $\mathcal Q\subseteq\R^d$, the function $\psi$ is said to
		satisfy the Polyak--\L{}ojasiewicz (PL) condition on $\mathcal Q$ if
		there exists $\mu_{\rm PL}>0$ such that
		\begin{equation}
			\label{eq:pl-definition}
			\frac12\norm{\grad\psi(u)}^2
			\ge
			\mu_{\rm PL}\bigl(\psi(u)-\psi^*\bigr),
			\qquad u\in\mathcal Q.
		\end{equation}
	\end{definition}
	
	\begin{remark}
		We can see that strong convexity implies the PL condition, whereas the PL condition itself does not require convexity in \cite{AbbaszadehpeivastiEtAl2023}.
	\end{remark}
	
	For the plus function
	\begin{equation}
		\label{eq:plus}
		\theta(r):=\max\{0,r\},
		\qquad r\in\R,
	\end{equation}
	we consider the hyperbolic smoothing function used in
	\cite{Bagirov2013}:
	\begin{equation}
		\label{eq:phi}
		\phi_\tau(r):=
		\frac{r+\sqrt{r^2+\tau^2}}{2},
		\qquad r\in\R,
	\end{equation}
	where $\tau>0$ is the smoothing parameter.
	
	\begin{proposition}[{\cite[Proposition~1]{Bagirov2013}}]
		\label{prop:phi-properties}
		The hyperbolic smoothing function $\phi_\tau$ has the following
		properties:
		\begin{enumerate}[label=\textup{(\roman*)},leftmargin=*]
			\item $\phi_\tau$ is monotonically increasing, convex, and infinitely
			differentiable;
			\item for every $r\in\R$,
			\begin{equation}
				\label{eq:phi-gap}
				\theta(r)<\phi_\tau(r)
				\le\theta(r)+\frac{\tau}{2}.
			\end{equation}
		\end{enumerate}
	\end{proposition}
	The epigraph formulation and its associated finite-max function are
	\begin{equation}
		\label{eq:epigraph}
		\min_{(x,t)\in\R^{n+1}}\{t:f_i(x)-t\le0,\ i\in\Iset\},
		\qquad
		F(x,t):=t+\sum_{i\in\Iset}\max\{0,f_i(x)-t\}.
	\end{equation}
	\begin{proposition}[{\cite[Proposition~2]{Bagirov2013}}]
		\label{prop:f-F-value}
		For every $x\in\R^n$,
		\begin{equation}
			\label{eq:f-F-value}
			f(x)=\min_{t\in\R}F(x,t),
			\qquad F(x,f(x))=f(x).
		\end{equation}
	\end{proposition}
	
	\begin{proposition}[{\cite[Proposition~3]{Bagirov2013}}]
		\label{prop:f-F-stationarity}
		Suppose that every $f_i$ is continuously differentiable.
		\begin{enumerate}[label=\textup{(\roman*)},leftmargin=*]
			\item If $x^*$ is Clarke stationary for $f$, then
			$(x^*,f(x^*))$ is Clarke stationary for $F$.
			\item If $(x^*,t^*)$ is Clarke stationary for $F$, then $x^*$ is
			Clarke stationary for $f$.
		\end{enumerate}
	\end{proposition}
	
	\begin{proposition}[{\cite[Proposition~4]{Bagirov2013}}]
		\label{prop:f-F-minimizer}
		\begin{enumerate}[label=\textup{(\roman*)},leftmargin=*]
			\item If $x^*$ is a local minimizer of $f$, then
			$(x^*,f(x^*))$ is a local minimizer of $F$.
			\item If $(x^*,t^*)$ is a local minimizer of $F$, then $x^*$ is a
			local minimizer of $f$.
		\end{enumerate}
	\end{proposition}
	% ============================================================
	\section{Hyperbolic Smoothing}
	\label{sec:hyperbolic}
	% ============================================================
	This section introduces the hyperbolic smoothing function of $F$,
	establishes its smoothness properties, and presents its approximation
	and stationarity relationships with $F$.
	
	Applying \eqref{eq:phi} to the plus terms in $F$ gives the hyperbolic
	smoothing function of $F$
	\begin{equation}
		\label{eq:full-smoothing}
		\Phi_\tau(x,t):=t+\sum_{i\in\Iset}\phi_\tau(f_i(x)-t).
	\end{equation}
	For $r_{i,\tau}(x,t):=f_i(x)-t$, define
	\[
	q_i(x):=\begin{pmatrix}\grad f_i(x)\\-1\end{pmatrix},\qquad
	D_i(x):=\begin{pmatrix}\hess f_i(x)&0\\0&0\end{pmatrix},\qquad
	\beta_{i,\tau}(x,t):=\frac{r_{i,\tau}(x,t)}
	{\sqrt{r_{i,\tau}(x,t)^2+\tau^2}},
	\]
	and
	\[
	\vartheta_{i,\tau}(x,t):=
	\frac{\tau^2}{\bigl(r_{i,\tau}(x,t)^2+\tau^2\bigr)^{3/2}}.
	\]
	Whenever the displayed derivatives exist, direct differentiation gives
	\begin{align}
		\label{eq:full-gradient}
		\grad\Phi_\tau(x,t)
		&=
		\begin{pmatrix}
			\displaystyle\frac12\sum_{i\in\Iset}
			\bigl(1+\beta_{i,\tau}(x,t)\bigr)\grad f_i(x)\\[1mm]
			\displaystyle 1-\frac{m}{2}-\frac12\sum_{i\in\Iset}
			\beta_{i,\tau}(x,t)
		\end{pmatrix},\\
		\label{eq:full-hessian}
		\hess\Phi_\tau(x,t)
		&=\frac12\sum_{i\in\Iset}
		\left[\bigl(1+\beta_{i,\tau}(x,t)\bigr)D_i(x)
		+\vartheta_{i,\tau}(x,t)q_i(x)q_i(x)^T\right].
	\end{align}
	
	According to~\cite[Proposition~5]{Bagirov2013}, the smoothing error satisfies
	\begin{equation}
		\label{eq:smoothing-gap}
		0<\Phi_\tau(x,t)-F(x,t)\le\frac{m\tau}{2}.
	\end{equation}
	
	\begin{proposition}[{\cite[Proposition~6]{Bagirov2013}}]
		\label{prop:fixed-gradient-consistency}
		Suppose that every $f_i$ is continuously differentiable.  For every fixed
		$(x,t)\in\R^{n+1}$, the limit
		\[
		v(x,t):=\lim_{\tau\downarrow0}\grad\Phi_\tau(x,t)
		\]
		exists and satisfies
		\[
		v(x,t)\in\clarke F(x,t).
		\]
	\end{proposition}
	
	\begin{proof}
		Set
		\[
		a_{i,\tau}(x,t):=
		\frac12\bigl(1+\beta_{i,\tau}(x,t)\bigr).
		\]
		Then \eqref{eq:full-gradient} can be written as
		\[
		\grad\Phi_\tau(x,t)
		=
		\left(
		\sum_{i\in\Iset}a_{i,\tau}(x,t)\grad f_i(x),\,
		1-\sum_{i\in\Iset}a_{i,\tau}(x,t)
		\right).
		\]
		Since $(x,t)$ is fixed,
		\[
		a_{i,\tau}(x,t)\longrightarrow
		\begin{cases}
			0, & f_i(x)<t,\\[1mm]
			\dfrac12, & f_i(x)=t,\\[1mm]
			1, & f_i(x)>t,
		\end{cases}
		\qquad \tau\downarrow0.
		\]
		On the other hand, for
		$\psi_i(x,t):=\max\{0,f_i(x)-t\}$,
		\[
		\clarke\psi_i(x,t)=
		\begin{cases}
			\{0_{n+1}\}, & f_i(x)<t,\\
			\textbf{co}\{q_i(x)\}, & f_i(x)=t,\\
			\{q_i(x)\}, & f_i(x)>t.
		\end{cases}
		\]
		Because
		$F(x,t)=t+\sum_{i\in\Iset}\psi_i(x,t)$,
		the limiting coefficients above select $\lambda=1/2$ whenever
		$f_i(x)=t$.  Hence the limit of $\grad\Phi_\tau(x,t)$ is an element of
		$\clarke F(x,t)$.
	\end{proof}
	
	\begin{proposition}[{\cite[Proposition~8]{Bagirov2013}}]
		\label{prop:gradient-local-lipschitz}
		Suppose that every $f_i$ is continuously differentiable and
		$\grad f_i$ is locally Lipschitz. Then, for every fixed $\tau>0$,
		$\grad\Phi_\tau$ is locally Lipschitz on $\R^{n+1}$.
	\end{proposition}
	
	\begin{proof}
		By \eqref{eq:full-gradient} and the definition of $q_i$,
		\[
		\grad\Phi_\tau(x,t)
		=
		\begin{pmatrix}0_n\\1\end{pmatrix}
		+
		\frac12\sum_{i\in\Iset}
		\bigl(1+\beta_{i,\tau}(x,t)\bigr)q_i(x).
		\]
		For fixed $\tau>0$, let
		\[
		h_\tau(r):=\frac{r}{\sqrt{r^2+\tau^2}}.
		\]
		Since
		\[
		|h_\tau'(r)|
		=\frac{\tau^2}{(r^2+\tau^2)^{3/2}}
		\le \frac1{\tau},
		\]
		$h_\tau$ is Lipschitz. Moreover,
		$(x,t)\mapsto f_i(x)-t$ is locally Lipschitz, and hence
		$\beta_{i,\tau}$ is locally Lipschitz.
		Since $\grad f_i$ is locally Lipschitz, it follows from
		\eqref{eq:full-gradient} that $\grad\Phi_\tau$ is locally Lipschitz.
	\end{proof}

	\begin{proposition}%[{\cite[Proposition~9]{Bagirov2013}}]
		\label{prop:hessian-bound}
		Suppose that every $f_i$ is twice continuously differentiable. Let
		$\mathcal Q\subset\R^{n+1}$ be bounded and set
		\[
		\widetilde{\mathcal Q}:=\co(\overline{\mathcal Q}),\qquad
		\gamma_{\mathcal Q}:=
		\max_{(x,t)\in\widetilde{\mathcal Q}}\gamma(x),\qquad
		\omega_{\mathcal Q}:=
		\max_{(x,t)\in\widetilde{\mathcal Q}}\omega(x).
		\]
		Then, for every fixed $\tau>0$,
		\begin{equation}
			\label{eq:tau-dependent-hessian-bound}
			\left|
			\left\langle
			u,\hess\Phi_\tau(z)u
			\right\rangle
			\right|
			\le
			L_{\mathcal Q,\tau}\norm{u}^2,
			\qquad
			z\in\widetilde{\mathcal Q},\quad u\in\R^{n+1},
		\end{equation}
		where
		\begin{equation}
			\label{eq:tau-dependent-L}
			L_{\mathcal Q,\tau}
			:=
			m\omega_{\mathcal Q}
			+\frac{m}{2\tau}
			\bigl(\gamma_{\mathcal Q}^2+1\bigr).
		\end{equation}
		In particular,
		\[
		\norm{\hess\Phi_\tau(z)}
		\le L_{\mathcal Q,\tau},
		\qquad z\in\widetilde{\mathcal Q}.
		\]
	\end{proposition}
	
	\begin{proof}
		Let $z=(x,t)\in\widetilde{\mathcal Q}$ and write
		$u=(v,\sigma)\in\R^n\times\R$. From
		\eqref{eq:full-hessian},
		\[
		\left\langle u,\hess\Phi_\tau(z)u\right\rangle
		=\frac12\sum_{i\in\Iset}
		\bigl(1+\beta_{i,\tau}(x,t)\bigr)
		\left\langle v,\hess f_i(x)v\right\rangle
		+\frac12\sum_{i\in\Iset}
		\vartheta_{i,\tau}(x,t)
		\left\langle q_i(x),u\right\rangle^2 .
		\]
		Because
		\[
		0<1+\beta_{i,\tau}(x,t)<2,
		\qquad
		0\le\vartheta_{i,\tau}(x,t)\le\frac1{\tau},
		\]
		and, by \eqref{eq:gamma-omega},
		\[
		\left|
		\left\langle v,\hess f_i(x)v\right\rangle
		\right|
		\le
		\omega_{\mathcal Q}\norm{v}^2,
		\]
		\[
		\left\langle q_i(x),u\right\rangle^2
		\le
		\norm{q_i(x)}^2\norm{u}^2
		\le
		\bigl(\gamma_{\mathcal Q}^2+1\bigr)\norm{u}^2,
		\]
		we obtain
		\[
		\left|
		\left\langle u,\hess\Phi_\tau(z)u\right\rangle
		\right|
		\le
		\left(
		m\omega_{\mathcal Q}
		+\frac{m}{2\tau}
		\bigl(\gamma_{\mathcal Q}^2+1\bigr)
		\right)\norm{u}^2.
		\]
		This proves \eqref{eq:tau-dependent-hessian-bound}. Since
		$\hess\Phi_\tau(z)$ is symmetric, the same constant bounds its
		operator norm.
	\end{proof}
	
	\begin{remark}
		Proposition~\ref{prop:hessian-bound} gives a uniform bound for the
		Hessian on bounded sets; it does not assert that $\hess\Phi_\tau$ itself
		is Lipschitz. The latter is a stronger local regularity condition and is
		imposed separately only where it is needed for the local superlinear
		BFGS analysis.
	\end{remark}
	
	\begin{corollary}
		\label{cor:stationarity-lift}
		Suppose that every $f_i$ is continuously differentiable. If
		$(x_j,t_j)\to(x^*,t^*)$, $\tau_j\downarrow0$, and
		\[
		\grad\Phi_{\tau_j}(x_j,t_j)\to0_{n+1},
		\]
		then
		\begin{equation}
			\label{eq:F-stationarity-lift}
			0_{n+1}\in\partial F(x^*,t^*).
		\end{equation}
		Consequently,
		\begin{equation}
			\label{eq:stationarity-lift}
			0_n\in\partial f(x^*).
		\end{equation}
	\end{corollary}
	
	\begin{proof}
		Set
		\[
		a_{i,j}:=
		\frac12\bigl(1+\beta_{i,\tau_j}(x_j,t_j)\bigr).
		\]
		Then
		\[
		\grad\Phi_{\tau_j}(x_j,t_j)
		=
		\begin{pmatrix}0_n\\1\end{pmatrix}
		+
		\sum_{i\in\Iset}
		a_{i,j}q_i(x_j).
		\]
		Passing to a subsequence if necessary, let
		$a_{i,j}\to a_i\in[0,1]$. From
		$(x_j,t_j)\to(x^*,t^*)$ and $\tau_j\downarrow0$,
		\[
		a_i=
		\begin{cases}
			0, & f_i(x^*)<t^*,\\
			[0,1], & f_i(x^*)=t^*,\\
			1, & f_i(x^*)>t^*,
		\end{cases}
		\]
		where the middle line means simply $a_i\in[0,1]$.
		These are exactly the admissible coefficients in the Clarke
		subdifferential of the plus term
		$\max\{0,f_i(x)-t\}$. Therefore
		\[
		\begin{pmatrix}0_n\\1\end{pmatrix}
		+\sum_{i\in\Iset}a_iq_i(x^*)
		\in\partial F(x^*,t^*).
		\]
		Taking limits in
		$\grad\Phi_{\tau_j}(x_j,t_j)\to0_{n+1}$ gives
		\eqref{eq:F-stationarity-lift}. Proposition~
		\ref{prop:f-F-stationarity}~(ii) then yields
		\eqref{eq:stationarity-lift}.
	\end{proof}
	
	% ============================================================
	\section{Hyperbolic Majorization Preconditioner}
	\label{sec:preconditioner}
	% ============================================================
	The curvature of the scalar smoothing term is largest at a switching
	surface: \(\phi_\tau''(0)=1/(2\tau)\).  Thus the Euclidean conditioning of
	\(\Phi_\tau\) can deteriorate when \(\tau\downarrow0\).  We use the following
	problem-derived metric in the L-BFGS initial matrix.  It is not an additional
	heuristic diagonal scaling.
	
	Let \(z=(x,t)\), \(r_i=f_i(x)-t\), \(\omega_i=(r_i^2+\tau^2)^{1/2}\), and
	let \(J(z)\in\R^{m\times(n+1)}\) have row \(q_i(x)^T\).  Put
	\[
	\Pi_x:=\begin{pmatrix}I_n&0\\0&0\end{pmatrix},\qquad
	D^{\rm hyp}(z,\tau):=\diag(\omega_1^{-1},\ldots,\omega_m^{-1}).
	\]
	Suppose that on a set containing the accepted iterates and every Armijo trial
	segment, \(\nabla f_i\) is \(L_i\)-Lipschitz, and set
	\(\bar L:=\sum_{i\in\Iset}L_i\).  For \(\delta>0\), define
	\begin{equation}
		\label{eq:Pk}
		P(z,\tau):=\delta I_{n+1}+\bar L\Pi_x+
		\frac12J(z)^TD^{\rm hyp}(z,\tau)J(z).
	\end{equation}
	At the \(\ell\)-th inner iteration of layer \(k\), we write
	\(P_k^\ell=P(z_k^\ell,\tau_k)\).  The code constructs exactly
	\eqref{eq:Pk}: its augmented Jacobian has columns
	\([\nabla f_i(x)^T,-1]^T\).
	
	\begin{lemma}
		\label{prop:preconditioner-spd}
		For every \(z\) and \(\tau>0\), \(P(z,\tau)\succeq\delta I_{n+1}\succ0\).
		If the components are twice continuously differentiable at \(x\), then
		\begin{equation}
			\label{eq:preconditioner-loewner}
			P(z,\tau)-\hess\Phi_\tau(z)\succeq\delta I_{n+1}.
		\end{equation}
	\end{lemma}
	
	\begin{proof}
		The two nonregularizing summands in \eqref{eq:Pk} are positive semidefinite,
		which proves the first assertion.  By \eqref{eq:full-hessian}, with
		\(a_i=(1+\beta_{i,\tau})/2\) and
		\(c_i=\tau^2/(2\omega_i^3)\),
		\[
		\hess\Phi_\tau(z)=\sum_i a_iD_i(x)+\sum_i c_iq_i(x)q_i(x)^T.
		\]
		Since \(0<a_i<1\), \(\hess f_i(x)\preceq L_iI\), and
		\(c_i\le1/(2\omega_i)\), the right-hand side is bounded above by
		\(\bar L\Pi_x+J^TD^{\rm hyp}J/2\).  Adding \(\delta I_{n+1}\) gives
		\eqref{eq:preconditioner-loewner}.
	\end{proof}
	
	\begin{lemma}[Scalar hyperbolic majorization]
		\label{lem:scalar-major}
		For \(u,v\in\R\) and \(\omega=(u^2+\tau^2)^{1/2}\),
		\begin{equation}
			\label{eq:scalar-major}
			\phi_\tau(u+v)\le\phi_\tau(u)+\phi_\tau'(u)v+\frac{v^2}{4\omega}.
		\end{equation}
	\end{lemma}
	
	\begin{proof}
		The concavity of the square-root function gives
		\[
		\sqrt{(u+v)^2+\tau^2}\le\omega+\frac{2uv+v^2}{2\omega}.
		\]
		Substitution into the definition of \(\phi_\tau\) proves the claim.
	\end{proof}
	
	\begin{theorem}
		\label{thm:global-major}
		Under the Lipschitz-gradient assumption used to define \(\bar L\), every
		\(d=(p,\eta)\in\R^{n+1}\) whose segment \(z+[0,1]d\) remains in that set
		satisfies
		\begin{equation}
			\label{eq:global-major}
			\Phi_\tau(z+d)\le\Phi_\tau(z)+\grad\Phi_\tau(z)^Td+
			\frac12d^TP(z,\tau)d.
		\end{equation}
	\end{theorem}
	
	\begin{proof}
		The component descent lemma yields
		\[
		f_i(x+p)-(t+\eta)\le r_i+q_i(x)^Td+\frac{L_i}{2}\norm p^2.
		\]
		Put \(a_i=(1+\beta_{i,\tau})/2=\phi_\tau'(r_i)\) and
		\(e_i=L_i\norm p^2/2\ge0\).  Since \(\phi_\tau\) is increasing and
		\(1\)-Lipschitz, the preceding inequality implies
		\[
		\phi_\tau(f_i(x+p)-t-\eta)
		\le \phi_\tau(r_i+q_i(x)^Td+e_i)
		\le \phi_\tau(r_i+q_i(x)^Td)+e_i.
		\]
		Applying Lemma~\ref{lem:scalar-major} to the last term (with
		\(u=r_i\) and \(v=q_i(x)^Td\)) gives
		\[
		\phi_\tau(f_i(x+p)-t-\eta)
		\le\phi_\tau(r_i)+a_iq_i(x)^Td+
		\frac{(q_i(x)^Td)^2}{4\omega_i}+\frac{L_i}{2}\norm p^2.
		\]
		Summing the inequalities, adding \(t+\eta\), and using \eqref{eq:Pk}
		proves \eqref{eq:global-major}.
	\end{proof}
	
	\begin{remark}[Scope of \(\bar L\)]
		\label{rem:Lbar-scope}
		For components with globally Lipschitz gradients, \eqref{eq:global-major} is
		global.  For quartic test components, a finite global \(\bar L\) does not
		exist; the same statement is valid on any bounded region covering the actual
		accepted points and tested line-search segments.  The numerical value assigned
		to \texttt{preconditionerL} therefore has this regional interpretation in
		such experiments.
	\end{remark}
	
	% ============================================================
	\section{Three Negative-Curvature BFGS Correction Strategies}
	\label{sec:three-corrections}
	% ============================================================
	
	This section formalizes the correction of a strict negative-curvature
	pair.  The construction starts from a constrained nearest-point problem:
	the corrected curvature vector is required to lie on a positive-secant
	hyperplane and is otherwise chosen to remain as close as possible to the
	observed vector in a prescribed geometry.  This yields the Euclidean and
	quasi-Newton-metric corrections as two instances of one variational
	principle.  Direct sign reflection is retained as a separate, inexpensive
	baseline rule.
	
	Let $s,y\in\R^n$, with $s\ne0$ and $s^Ty<0$, and set
	\begin{equation}
		\label{eq:negative-notation}
		\kappa:=-s^Ty=|s^Ty|>0.
	\end{equation}
	The admissible corrected vectors therefore form the affine hyperplane
	\begin{equation}
		\label{eq:target-curvature}
		\mathcal H_\kappa:=\{v\in\R^n:s^Tv=\kappa\}.
	\end{equation}
	For a symmetric positive-definite matrix $M$, write
	$\|u\|_M^2:=u^TMu$.  We first consider the metric projection of $y$
	onto $\mathcal H_\kappa$:
	\begin{equation}
		\label{eq:weighted-problem}
		\ytilde(M):=\argmin_{v\in\mathcal H_\kappa}
		\frac12\|v-y\|_M^2.
	\end{equation}
	
	\begin{proposition}[Metric projection onto the positive-secant hyperplane]
		\label{prop:weighted-nearest}
		Let $M\in\R^{n\times n}$ be symmetric positive definite.  Problem
		\eqref{eq:weighted-problem} has a unique solution, and it is given by
		\begin{equation}
			\label{eq:weighted-correction}
			\ytilde(M)=y-\frac{2s^Ty}{s^TM^{-1}s}M^{-1}s.
		\end{equation}
		In particular, $s^T\ytilde(M)=\kappa>0$.
	\end{proposition}
	
	\begin{proof}
		The feasible set is nonempty and affine, while the objective in
		\eqref{eq:weighted-problem} is strictly convex.  Hence, the
		Karush--Kuhn--Tucker conditions are necessary and sufficient.  With
		$\lambda$ denoting the multiplier for the equality constraint, they are
		\[
		M(v-y)+\lambda s=0,
		\qquad s^Tv=\kappa.
		\]
		The first relation gives $v=y-\lambda M^{-1}s$.  Substitution into the
		constraint and use of $\kappa=-s^Ty$ yield
		\[
		\lambda=\frac{s^Ty-\kappa}{s^TM^{-1}s}
		=\frac{2s^Ty}{s^TM^{-1}s},
		\]
		which proves \eqref{eq:weighted-correction}.  The asserted secant
		identity follows from feasibility, and uniqueness follows from strict
		convexity.
	\end{proof}
	
	\subsection{Intrinsic metric choices}
	\label{subsec:weighted-nearest}
	
	The two natural geometries in the present setting are the Euclidean metric
	and the inverse quasi-Newton metric.  The following specializations make
	their respective correction directions explicit.
	
	\begin{corollary}[Euclidean and quasi-Newton-metric corrections]
		\label{cor:metric-specializations}
		Let $B\in\R^{n\times n}$ be symmetric positive definite and put
		$b:=s^TBs$.  The metric projection
		\eqref{eq:weighted-problem} gives
		\begin{align}
			\label{eq:euclidean-correction}
			\ytilde^{\Ecorr}
			&:=\ytilde(I)
			=y-\frac{2s^Ty}{s^Ts}s
			=y+\frac{2\kappa}{\|s\|^2}s,\\
			\label{eq:metric-correction}
			\ytilde^{\Hcorr}
			&:=\ytilde(B^{-1})
			=y-\frac{2s^Ty}{s^TBs}Bs
			=y+\frac{2\kappa}{b}Bs.
		\end{align}
		Thus, $\ytilde^{\Ecorr}-y\in\operatorname{span}\{s\}$ and
		$\ytilde^{\Hcorr}-y\in\operatorname{span}\{Bs\}$, and both
		vectors belong to $\mathcal H_\kappa$.
	\end{corollary}
	
	\begin{proof}
		Apply Proposition~\ref{prop:weighted-nearest} with $M=I$ and
		$M=B^{-1}$, respectively.
	\end{proof}
	
	The Euclidean correction is the ordinary minimum-change correction.  Since
	$s^Ty=-\kappa$, it can equivalently be written as
	\begin{equation}
		\label{eq:householder}
		\ytilde^{\Ecorr}=\left(I-2\frac{ss^T}{s^Ts}\right)y,
	\end{equation}
	which is the reflection of $y$ across $\{v:s^Tv=0\}$.  In contrast,
	$\ytilde^{\Hcorr}$ is the minimum-change correction in the geometry
	induced by the current inverse quasi-Newton model; it displaces $y$ in the
	direction $Bs$.
	
	\subsection{Direct sign reflection}
	\label{subsec:sign}
	
	For comparison, the direct sign-reflection rule is
	\begin{equation}
		\label{eq:sign-correction}
		\ytilde^{\Ncorr}:=-y.
	\end{equation}
	It satisfies $s^T\ytilde^{\Ncorr}=\kappa>0$ at the cost of only a sign
	change.  Unlike the two intrinsic metric choices above, it reverses every
	component of $y$, including those orthogonal to $s$.
	
	\begin{remark}[Auxiliary metric interpretation]
		For each strict negative-curvature pair, direct reflection can also be
		realized as the solution of \eqref{eq:weighted-problem} for a suitably
		adapted positive-definite metric depending on $(s,y)$.  Its explicit form
		is neither needed nor used in the algorithm.  We therefore keep
		\eqref{eq:sign-correction} as a separate baseline rule and reserve the
		variational construction for the intrinsic choices $M=I$ and $M=B^{-1}$.
	\end{remark}
	
	\subsection{Compatibility with the BFGS update}
	\label{subsec:unified-corrections}
	
	Once a corrected vector has been selected, all three variants use the same
	BFGS update and the same limited-memory recursion.  The next proposition
	records the common consequence of the positive-secant construction.
	
	\begin{proposition}
		\label{prop:three-common}
		Let $B\in\R^{n\times n}$ be symmetric positive definite.  Each of
		$\ytilde^{\Ncorr}$, $\ytilde^{\Ecorr}$, and $\ytilde^{\Hcorr}$ satisfies
		\begin{equation}
			\label{eq:three-common-curvature}
			s^T\ytilde=\kappa=|s^Ty|>0.
		\end{equation}
		Consequently,
		\begin{equation}
			\label{eq:bfgs}
			B^+(\ytilde)
			=B-\frac{Bss^TB}{s^TBs}
			+\frac{\ytilde\ytilde^T}{s^T\ytilde}
		\end{equation}
		is symmetric positive definite and satisfies $B^+(\ytilde)s=\ytilde$.
		Equivalently, for $H=B^{-1}$ and $\rho=(s^T\ytilde)^{-1}$,
		\begin{equation}
			\label{eq:inverse-bfgs}
			H^+(\ytilde)
			=(I-\rho s\ytilde^T)H(I-\rho\ytilde s^T)+\rho ss^T.
		\end{equation}
	\end{proposition}
	
	\begin{proof}
		The claim for $\ytilde^{\Ncorr}$ follows directly from
		\eqref{eq:sign-correction}; the other two follow from
		Corollary~\ref{cor:metric-specializations}.  Hence
		$s^T\ytilde>0$ in every case.  The standard BFGS positive-definiteness
		and secant properties then give the conclusions in
		\eqref{eq:bfgs}--\eqref{eq:inverse-bfgs}.
	\end{proof}
	
	We use the two stored-pair safeguards themselves to organize the
	pair-screening rule. For a nonzero secant pair $(s,y)$, let
	\begin{equation}
		\label{eq:curvature-classification}
		c:=s^Ty,\qquad
		\eta(s,y):=
		\max\left\{
		\deltac\norm{s}^2,\,
		\frac{\norm{y}^2}{\Lc}
		\right\}.
	\end{equation}
	Every stored pair is required to satisfy
	\begin{equation}
		\label{eq:pair-safeguard}
		\frac{s^T\ytilde}{\norm{s}^2}>\deltac,
		\qquad
		\frac{\norm{\ytilde}^2}{s^T\ytilde}<\Lc,
		\qquad 0<\deltac<\Lc.
	\end{equation}
	If $c>\eta(s,y)$, the original positive-curvature pair already satisfies
	both safeguards and is stored without correction. For the $\Ncorr$ and
	$\Ecorr$ corrections, $s^T\ytilde=|c|$ and
	$\norm{\ytilde}=\norm{y}$; hence $c<-\eta (s,y)$ guarantees both
	safeguards before the correction is formed. The $\Hcorr$ correction is
	handled in two stages. Its first-stage test uses only
	$c<-\deltac\norm{s}^2$, which guarantees the first safeguard after
	correction because $s^T\ytilde^{\Hcorr}=|c|$. Only then is
	$\mathcal B_{\mathcal M}s$ evaluated and the corrected vector formed; the
	second safeguard is checked on $\ytilde^{\Hcorr}$ itself. Thus the
	$H$-metric branch does not reject a pair merely because the uncorrected
	Euclidean norm $\norm{y}$ fails the second safeguard.
	
	% ============================================================
	% The implementation-consistent preconditioned development replaces the
	% superseded scalar-initial-metric development above.
	\section{Preconditioned Limited-Memory BFGS Implementation}
	\label{sec:lbfgs}
	% ============================================================
	Let \(m_L\ge1\) be the memory capacity and let
	\[
	\mathcal M=\{(s_j,\widetilde y_j,\rho_j)\}_{j=1}^q,
	\qquad q\le m_L,\qquad
	\rho_j=(s_j^T\widetilde y_j)^{-1}.
	\]
	All stored pairs obey \(s_j^T\widetilde y_j>0\) and the screening safeguards
	\eqref{eq:pair-safeguard}.  At the current point \(z\), the implementation
	uses the current matrix \(P=P(z,\tau)\) in the initial inverse metric:
	\begin{equation}
		\label{eq:preconditioned-initial-scaling}
		H^{(0)}=\chi P^{-1},\qquad
		\chi=\begin{cases}
			\chi_0, & q=0,\\[1mm]
			\displaystyle\clip_{[\chi_{\min},\chi_{\max}]}
			\left(\frac{s_q^T\widetilde y_q}
			{\widetilde y_q^TP^{-1}\widetilde y_q}\right), & q>0.
		\end{cases}
	\end{equation}
	Thus the scale in the code is the BB2 scale in the current \(P^{-1}\)-metric,
	not the Euclidean expression in the unpreconditioned method.  For the H-type
	negative-curvature correction, the direct compact recurrence is initialized by
	\begin{equation}
		\label{eq:preconditioned-direct-initial}
		B^{(0)}=(H^{(0)})^{-1}=P/\chi.
	\end{equation}
	The two formulas are exact inverses.  Turning the preconditioner off sets
	\(P=I\) and recovers the original implementation.
	
	The preconditioned metric in \eqref{eq:preconditioned-initial-scaling} must be
	translated into an efficient limited-memory direction computation.  For this
	purpose, Algorithm~\ref{alg:P-twoloop} records the preconditioned two-loop
	recursion used at every inner HMLBFGS iteration.  It is not a separate
	optimization method, but the computational subroutine that maps the current
	gradient to a search direction; compared with the standard L-BFGS recursion,
	the essential change is the application of \(\chi P^{-1}\) in the middle
	scaling step.
	
	\begin{algorithm}[htbp]
		\caption{Preconditioned L-BFGS two-loop recursion}
		\label{alg:P-twoloop}
		\begin{algorithmic}[1]
			\Require Gradient \(g\), current \(P\succ0\), and ordered memory
			\(\mathcal M=\{(s_j,\widetilde y_j,\rho_j)\}_{j=1}^{q}\)
			\Ensure \(d=-H_{\mathcal M,P}g\), without changing \(\mathcal M\)
			\State Set \(v\gets g\)
			\For{\(j=q,q-1,\ldots,1\)}
			\State \(a_j\gets\rho_j s_j^Tv\); \(v\gets v-a_j\widetilde y_j\)
			\EndFor
			\State Set \(\chi\) by \eqref{eq:preconditioned-initial-scaling} and
			\(r\gets\chi P^{-1}v\)
			\For{\(j=1,2,\ldots,q\)}
			\State \(b_j\gets\rho_j\widetilde y_j^Tr\); \(r\gets r+s_j(a_j-b_j)\)
			\EndFor
			\State \Return \(d=-r\)
		\end{algorithmic}
	\end{algorithm}
	
	For the H correction, \(B_{\mathcal M,P}v\) is evaluated without forming a
	dense matrix.  Starting from \(B^{(0)}v=Pv/\chi\), propagate the stored
	vectors and \(v\) through
	\begin{equation}
		\label{eq:preconditioned-direct-product}
		B^{(j)}=B^{(j-1)}-
		\frac{B^{(j-1)}s_js_j^TB^{(j-1)}}{s_j^TB^{(j-1)}s_j}
		+\frac{\widetilde y_j\widetilde y_j^T}{s_j^T\widetilde y_j},
		\qquad B_{\mathcal M,P}v:=B^{(q)}v.
	\end{equation}
	This is exactly the code path used to form \(B_{\mathcal M,P}s\) in the
	H-type correction.  It does not change the memory.
	
	\begin{lemma}[Positive definiteness and fixed-layer metric bounds]
		\label{lem:lbfgs-uniform}
		Fix \(\tau>0\) and a compact set \(\mathcal K\) containing the inner
		iterates.  Suppose the component gradients are continuous on the projection of
		\(\mathcal K\), every stored pair obeys \eqref{eq:pair-safeguard}, and
		\(q\le m_L\).  Then there are constants
		\(0<\underline h_\tau\le\overline h_\tau<\infty\), independent of the
		inner iteration, such that
		\begin{equation}
			\label{eq:preconditioned-metric-bounds}
			\underline h_\tau I\preceq H_{\mathcal M,P}\preceq\overline h_\tau I.
		\end{equation}
		In particular, the L-BFGS direction is a strict descent direction whenever
		\(g\ne0\).
	\end{lemma}
	
	\begin{proof}
		Continuity and compactness give
		\(\overline p_\tau:=\max_{z\in\mathcal K}\lambda_{\max}(P(z,\tau))<\infty\),
		while Proposition~\ref{prop:preconditioner-spd} gives \(P\succeq\delta I\).
		Set \(\chi_-:=\min\{\chi_0,\chi_{\min}\}\) and
		\(\chi_+:=\max\{\chi_0,\chi_{\max}\}\).  From
		\eqref{eq:preconditioned-direct-initial},
		\[
		\lambda_{\max}(B^{(0)})\le\overline p_\tau/\chi_-.
		\]
		The direct BFGS update and \eqref{eq:pair-safeguard} imply
		\[
		\lambda_{\max}(B^{(j+1)})
		\le\lambda_{\max}(B^{(j)})+
		\frac{\norm{\widetilde y_j}^2}{s_j^T\widetilde y_j}
		\le\lambda_{\max}(B^{(j)})+\Lc.
		\]
		Consequently one can take
		\[
		\underline h_\tau=
		\left(\overline p_\tau/\chi_-+m_L\Lc\right)^{-1}.
		\]
		For the upper bound, \(\norm{H^{(0)}}\le\chi_+/\delta\).  The inverse BFGS
		formula and the two screening inequalities give
		\[
		\norm{H^{(j+1)}}
		\le\left(1+\frac{\Lc}{\deltac}\right)^2\norm{H^{(j)}}+
		\frac1{\deltac}.
		\]
		Iteration of this recurrence defines a finite \(\overline h_\tau\).  The
		positive definiteness follows from \(H^{(0)}\succ0\) and the positive
		curvature of every stored pair.  Hence
		\(g^T(-H_{\mathcal M,P}g)<0\) for \(g\ne0\).
	\end{proof}
	
	\begin{remark}[Why the bounds are fixed-layer bounds]
		\label{rem:no-uniform-tau}
		The constant \(\overline p_\tau\) may grow as \(\tau\downarrow0\), because
		\(D^{\rm hyp}\) has entries of order \(\tau^{-1}\) near a switching surface.
		The code-level pair safeguards are Euclidean, and the two-loop direction is
		only required to pass the descent check.  Therefore the present implementation
		justifies uniform bounds for each fixed smoothing layer, but not a
		\(\tau\)-independent L-BFGS energy inequality.  In particular, one must not
		claim a \(\tau\)-uniform Armijo lower bound from \eqref{eq:global-major}
		without an additional \(P\)-metric pair safeguard or an explicit energy test.
	\end{remark}
	
	\begin{remark}[Large-scale application of \(P^{-1}\)]
		Write \(A=\delta I+\bar L\Pi_x\) and
		\(U=2^{-1/2}(D^{\rm hyp})^{1/2}J\).  Then \(P=A+U^TU\), and
		\begin{equation}
			\label{eq:woodbury}
			P^{-1}=A^{-1}-A^{-1}U^T(I+UA^{-1}U^T)^{-1}UA^{-1}.
		\end{equation}
		This is the Woodbury realization used by the code when the number of
		components is small; its middle system has order \(m\).  For sparse,
		large-component structures, the identical matrix is applied by a sparse direct
		solve.  These are exact linear-algebra realizations of the same \(P\).
	\end{remark}
	
	% ============================================================
	\section{Preconditioned HMLBFGS Algorithms}
	\label{sec:algorithm}
	% ============================================================
	Let \(0<\tau_{\rm fac}<1\), \(\tau_0>0\), and set
	\begin{equation}
		\label{eq:geometric-tau}
		\tau_k=\tau_0\tau_{\rm fac}^{k},\qquad k=0,1,\ldots .
	\end{equation}
	This geometric continuation is the one implemented in
	\texttt{hmlbfgs-pee-royset.m}; it replaces the former \(\tau_k=1/p_k\)
	description.  The Armijo condition is
	\begin{equation}
		\label{eq:full-armijo}
		\Phi_{\tau_k}(z+\lambda d)\le\Phi_{\tau_k}(z)+
		\alpha\lambda\grad\Phi_{\tau_k}(z)^Td,
		\qquad 0<\alpha<1.
	\end{equation}
	
	Before giving the complete HMLBFGS scheme, we isolate the treatment of a new
	secant pair.  This separation is useful because the three proposed variants
	share the same continuation, preconditioning, line search, and L-BFGS
	recursion, and differ only in the treatment of strict negative curvature.
	Algorithm~\ref{alg:correction} therefore collects the screening, correction,
	and storage decisions in one subroutine: it keeps an admissible positive pair,
	modifies a strict negative-curvature pair according to
	\(\Cmode\in\{\Ncorr,\Ecorr,\Hcorr\}\), or discards the pair if the
	prescribed safeguards are not satisfied.
	
	\begin{algorithm}[htbp]
		\caption{Secant-pair screening and correction used by HMLBFGS}
		\label{alg:correction}
		\begin{algorithmic}[1]
			\Require \(s\ne0\), \(y\), current \(P\), memory \(\mathcal M\), and
			mode \(\Cmode\in\{\Ncorr,\Ecorr,\Hcorr\}\)
			\Ensure either \((\widetilde y,\texttt{store})\) or \((\varnothing,\texttt{skip})\)
			\State Set \(c\gets s^Ty\), \(\eta\gets\max\{\deltac\norm{s}^2,\norm y^2/\Lc\}\)
			\If{the optional numerical curvature tolerance classifies \(c\) as zero}
			\State \Return \((\varnothing,\texttt{skip})\)
			\EndIf
			\If{\(c>\eta\)}
			\State \Return \((y,\texttt{store})\)
			\EndIf
			\If{\(\Cmode=\Ncorr\) and \(c<-\eta\)}
			\State \Return \((-y,\texttt{store})\)
			\ElsIf{\(\Cmode=\Ecorr\) and \(c<-\eta\)}
			\State \Return \((y-2c\,s/(s^Ts),\texttt{store})\)
			\ElsIf{\(\Cmode=\Hcorr\) and \(c<-\deltac\norm{s}^2\)}
			\State \(b\gets B_{\mathcal M,P}s\), \(b_s\gets s^Tb\)
			\If{\(b_s>0\)}
			\State \(\widetilde y\gets y-2c\,b/b_s\)
			\If{\(\norm{\widetilde y}^2<\Lc|c|\)}
			\State \Return \((\widetilde y,\texttt{store})\)
			\EndIf
			\EndIf
			\EndIf
			\State \Return \((\varnothing,\texttt{skip})\)
		\end{algorithmic}
	\end{algorithm}
	\FloatBarrier
	
	With the two computational building blocks specified, the full method can now
	be presented.  Algorithm~\ref{alg:main} combines geometric smoothing
	continuation, construction of the hyperbolic-majorization preconditioner, the
	preconditioned two-loop recursion in Algorithm~\ref{alg:P-twoloop}, Armijo
	backtracking, and the pair-processing step in Algorithm~\ref{alg:correction}.
	It serves as the common template for all three methods; choosing
	\(\Cmode=\Ncorr\), \(\Ecorr\), or \(\Hcorr\) produces the corresponding
	HMLBFGS variant while leaving the remaining steps unchanged.
	
	\begin{algorithm}[htbp]
		\caption{Preconditioned HMLBFGS with geometric smoothing continuation}
		\label{alg:main}
		\begin{algorithmic}[1]
			\Require \(x_0\), \(\tau_0\), \(\tau_{\rm fac}\), \(m_L\), \(\Cmode\),
			\(\delta,\bar L\), \(\alpha,\beta\in(0,1)\), scale and screening
			parameters, \texttt{outerMax}, \texttt{innerMax}, and tolerance \texttt{tol}
			\State \(z\gets(x_0,f(x_0))\)
			\For{\(k=0,1,\ldots,\texttt{outerMax}-1\)}
			\State \(\tau\gets\tau_0\tau_{\rm fac}^{k}\), \(\mathcal M\gets\varnothing\)
			\For{\(\ell=0,1,\ldots,\texttt{innerMax}-1\)}
			\State Evaluate \(\Phi_\tau(z)\), \(g=\grad\Phi_\tau(z)\), component gradients,
			and residuals
			\If{\(\norm g_\infty\le\texttt{tol}\)} \State declare this layer converged; \textbf{break} \EndIf
			\State Build \(P=P(z,\tau)\) by \eqref{eq:Pk} and obtain \(d\) from
			Algorithm~\ref{alg:P-twoloop}
			\If{\(d\) is nonfinite or \(g^Td\ge0\)}
			\State \(\mathcal M\gets\varnothing\), \(d\gets-g\)
			\EndIf
			\State Starting from \texttt{initialStep}, backtrack by \(\beta\) until
			\eqref{eq:full-armijo} holds; on a permitted restart, clear \(\mathcal M\)
			and retry with \(d=-g\)
			\If{no accepted step is found within the prescribed safeguard caps} \State stop with line-search failure \EndIf
			\State \(z^+\gets z+\lambda d\), \(g^+\gets\grad\Phi_\tau(z^+)\),
			\(s\gets z^+-z\), \(y\gets g^+-g\)
			\State Apply Algorithm~\ref{alg:correction} with the \emph{current} \(P\)
			and current memory; if it returns \texttt{store}, append
			\((s,\widetilde y,(s^T\widetilde y)^{-1})\), deleting the oldest pair if
			necessary
			\State \(z\gets z^+\)
			\EndFor
			\If{the layer hits \texttt{innerMax} and \texttt{innerMaxAction=stop}} \State stop \EndIf
			\EndFor
			\State \Return \(z\)
		\end{algorithmic}
	\end{algorithm}
	
	\FloatBarrier
	The modular form above also clarifies the convergence analysis: the
	preconditioned recursion provides the descent metric, while the pair-processing
	rule enforces the curvature safeguards needed by the limited-memory model.
	We now use these properties to analyze the complete scheme.
	
	% ============================================================
	\section{Convergence Properties}
	\label{sec:global}
	% ============================================================
	The statements in this section distinguish the actual finite-parameter code
	from its idealized continuation limit.  The code tests
	\(\norm g_\infty\le\texttt{tol}\) at each layer; it does not use a
	vanishing tolerance by default.
	
	\begin{assumption}[Compact level set]
		\label{ass:level-bounded}
		For the initial layer \(\tau_0\), the sublevel set
		\(\mathcal K_0=\{z:\Phi_{\tau_0}(z)\le\Phi_{\tau_0}(z_0)\}\) is compact.
	\end{assumption}
	
	\begin{lemma}[Fixed-layer descent and Armijo acceptance]
		\label{lem:armijo-reset}
		Fix \(\tau>0\), and suppose a layer remains in a compact subset of
		\(\mathcal K_0\).  Under the hypotheses of Lemma~\ref{lem:lbfgs-uniform},
		there are \(h_\tau,H_\tau>0\) such that each accepted nonrestart direction
		satisfies
		\[
		g^Td\le-h_\tau\norm g^2,\qquad \norm d\le H_\tau\norm g.
		\]
		The same inequalities hold after replacing \(h_\tau\) by
		\(\min\{h_\tau,1\}\) and \(H_\tau\) by \(\max\{H_\tau,1\}\) for the
		restart direction \(-g\).  Consequently, ordinary Armijo backtracking
		terminates and accepts a step bounded away from zero within this fixed layer.
	\end{lemma}
	
	\begin{proof}
		The first two inequalities follow from
		\eqref{eq:preconditioned-metric-bounds}.  On a compact neighborhood of the
		layer, \(\grad\Phi_\tau\) has a finite Lipschitz constant \(L_\tau\).
		The usual descent lemma shows that Armijo holds for every
		\[
		0<\lambda\le
		\min\left\{1,\frac{2(1-\alpha)h_\tau}{L_\tau H_\tau^2}\right\}.
		\]
		Geometric backtracking therefore terminates with a positive fixed-layer lower
		bound.  Notice that this proof uses the fixed-layer metric bounds, not an
		unavailable \(\tau\)-uniform energy estimate.
	\end{proof}
	
	\begin{theorem}[Fixed-layer first-order convergence]
		\label{thm:fixed-tau}
		Suppose the component regularity required in Lemma~\ref{lem:lbfgs-uniform}
		holds, all iterates and trial segments of the fixed layer remain in a compact
		set, and Armijo backtracking is uncapped. Then the inner iteration for a
		fixed \(\tau>0\) cannot generate infinitely many iterates
		with \(\norm{\grad\Phi_\tau(z)}\ge\varepsilon\) for any \(\varepsilon>0\).
		Thus \(\norm{\grad\Phi_\tau(z^\ell)}\to0\) if the inner loop is continued
		indefinitely.  With a positive stopping tolerance it reaches an
		\(\varepsilon\)-stationary point in finitely many iterations.
	\end{theorem}
	
	\begin{proof}
		Lemma~\ref{lem:armijo-reset} gives a fixed positive decrease whenever
		\(\norm g\ge\varepsilon\).  This contradicts the lower boundedness of
		\(\Phi_\tau\) on the compact level set if infinitely many such iterates occur.
	\end{proof}
	
	\begin{theorem}[Clarke stationarity for the ideal continuation]
		\label{thm:global}
		Suppose Assumption~\ref{ass:level-bounded} holds, every \(f_i\) is
		continuously differentiable, and the idealized continuation uses
		\eqref{eq:geometric-tau}.  Let \(z_k\) be a terminal point of layer \(k\)
		with \(\norm{\grad\Phi_{\tau_k}(z_k)}\le\varepsilon_k\), where
		\(\varepsilon_k\downarrow0\).  If the terminal points have an accumulation
		point \((x^*,t^*)\), then
		\[0\in\clarke f(x^*).\]
	\end{theorem}
	
	\begin{proof}
		Because \(0<\tau_{\rm fac}<1\), \(\tau_k\downarrow0\).  At a terminal
		point, \(\norm{\grad\Phi_{\tau_k}(z_k)}\le\varepsilon_k\to0\).  The
		stationarity lifting corollary \ref{cor:stationarity-lift} applies to every
		convergent subsequence and gives the result.
	\end{proof}
	
	\begin{remark}[What the current code certifies]
		With fixed \texttt{tol} and finite \texttt{outerMax}, the implementation
		returns a numerical \texttt{tol}-stationary point of the final smoothed layer
		(provided that layer reaches its stopping test).  This is not, by itself, an
		exact certificate of Clarke stationarity for the original nonsmooth minimax
		problem.  Theorem~\ref{thm:global} requires both an unbounded geometric
		continuation and \(\varepsilon_k\to0\); it is the appropriate asymptotic
		conclusion.
	\end{remark}
	
	% ============================================================
	\section{Fixed-Smoothing Linear Rates and Local BFGS Behavior}
	\label{sec:lbfgs-linear}
	\label{sec:local}
	% ============================================================
	For fixed \(\tau\), let \(\Phi_\tau^*=\inf_z\Phi_\tau(z)\).  The constants
	below are those of Lemma~\ref{lem:armijo-reset} for the fixed layer.
	
	\begin{theorem}[Q-linear objective convergence under PL]
		\label{thm:pl-linear}
		Fix \(\tau>0\).  Suppose that, from some index onward, the iterates and the
		trial segments stay in a set on which \(\Phi_\tau\) has an \(L_\tau\)-Lipschitz
		gradient and satisfies
		\[
		\frac12\norm{\grad\Phi_\tau(z)}^2\ge
		\mu_{\rm PL}\bigl(\Phi_\tau(z)-\Phi_\tau^*\bigr),\qquad \mu_{\rm PL}>0.
		\]
		Then the preconditioned HMLBFGS objective values converge Q-linearly:
		\[
		\Phi_\tau(z^{\ell+1})-\Phi_\tau^*
		\le q_{\rm PL}\bigl(\Phi_\tau(z^\ell)-\Phi_\tau^*\bigr),
		\qquad q_{\rm PL}\in(0,1).
		\]
	\end{theorem}
	
	\begin{proof}
		The proof of Lemma~\ref{lem:armijo-reset} supplies a fixed positive lower
		bound \(\widehat\lambda_\tau\) for the accepted step.  Armijo decrease and
		the PL inequality yield the claim with
		\(q_{\rm PL}=1-2\alpha\widehat\lambda_\tau h_\tau\mu_{\rm PL}\), after
		decreasing \(\widehat\lambda_\tau\), if needed, to make this number lie in
		\((0,1)\).
	\end{proof}
	
	\begin{theorem}[R-linear iterates under local strong convexity]
		\label{thm:strong-rlinear}
		Fix \(\tau>0\).  Let \(U\) be a convex neighborhood and suppose that
		\(z^*\in U\) is stationary and that, for every \(z\in U\),
		
		\[m_\Phi I\preceq\hess\Phi_\tau(z)\preceq M_\Phi I,qquad 0<m_\Phi\le M_\Phi.
		\]
		Assume that, from some index onward, all iterates and all trial segments
		remain in \(U\).  Then \(z^*\) is the unique minimizer of \(\Phi_\tau\) on \(U\),
		and the preconditioned HMLBFGS objective values converge Q-linearly to
		\(\Phi_\tau(z^*)\); consequently, the iterates converge R-linearly to \(z^*\).
	\end{theorem}
	
	\begin{proof}
		The Hessian bounds imply that \(\Phi_\tau\) is \(m_\Phi\)-strongly convex on
		\(U\).
		Since \(\grad\Phi_\tau(z^*)=0\), \(z^*\) is the unique minimizer on
		\(U\), and the local PL and quadratic-growth inequalities are
		\[
		\frac12\norm{\grad\Phi_\tau(z)}^2\ge m_\Phi\bigl(\Phi_\tau(z)-\Phi_\tau(z^*)\bigr),
		\qquad
		\Phi_\tau(z)-\Phi_\tau(z^*)\ge \frac{m_\Phi}{2}\norm{z-z^*}^2.
		\]
		Using the fixed-layer Armijo lower bound in Lemma~\ref{lem:armijo-reset},
		there is a \(q\in(0,1)\) such that, for all sufficiently large \(\ell\),
		\[
		\Phi_\tau(z^{\ell+1})-\Phi_\tau(z^*)
		\le q\bigl(\Phi_\tau(z^\ell)-\Phi_\tau(z^*)\bigr).
		\]
		The quadratic-growth inequality then gives
		\(\norm{z^\ell-z^*}\le C(\sqrt q)^{\,\ell}\) for a finite \(C\),
		which is R-linear convergence.
	\end{proof}
	
	\begin{proposition}[Local recovery of ordinary BFGS]
		\label{prop:local-inactive}
		Consider the full-memory BFGS method with the fixed initial matrix
		\(H^{(0)}=\chi_0P(z^0,\tau)^{-1}\); the preconditioner is not reinitialized
		at later iterations.  Suppose the hypotheses of Theorem~\ref{thm:strong-rlinear}
		hold, all sufficiently late accepted segments lie in its neighborhood, the
		Hessian is locally Lipschitz, and
		\(0<\deltac<m_\Phi\), \(\Lc>M_\Phi^2/m_\Phi\).  Then all sufficiently late
		pairs pass the direct positive-curvature test and the N, E, and H corrections
		are inactive.  If unit steps are eventually accepted, the corresponding
		full-memory modified BFGS method converges Q-superlinearly.
	\end{proposition}
	
	\begin{proof}
		For sufficiently late steps, \(y_k=G_ks_k\) with
		\(m_\Phi I\preceq G_k\preceq M_\Phi I\).  Hence
		\(s_k^Ty_k\ge m_\Phi\norm{s_k}^2\) and
		\(\norm{y_k}^2\le(M_\Phi^2/m_\Phi)s_k^Ty_k\), so the direct test is passed.
		The resulting tail is ordinary BFGS. More specifically, local strong
		convexity implies that, for all sufficiently large \(k\),
		\[
		s_k^Ty_k
		=
		s_k^T
		\left(
		\int_0^1\nabla^2\Phi_\tau(z_k+\theta s_k)\,d\theta
		\right)s_k
		\geq
		m_\Phi\|s_k\|^2>0,
		\]
		so that none of the negative-curvature corrections is activated eventually.
		Hence the modified secant vector coincides with the original one,
		\(\widetilde y_k=y_k\), and the method reduces locally to the standard BFGS
		iteration. The continuity of the Hessian yields the Dennis--Mor\'e condition,
		and the Dennis--Mor\'e theorem then gives the claimed superlinear conclusion
		under eventual unit steps.
	\end{proof}
	
	\section{Numerical Experiments}
	\label{sec:numerics}
	% ============================================================
	
	This section fixes the comparison protocol before any performance conclusion
	is drawn. Every reported run uses the same initial point, the same stopping
	tolerance, and the same evaluation-count convention. In particular, the
	three proposed methods use identical ordinary hyperbolic smoothing, continuation,
	reset, Armijo, and memory settings; their only difference is the strict
	negative-curvature branch in Algorithm~\ref{alg:correction}.
	
	\subsection{Methods and abbreviations}
	
	Table~\ref{tab:methods} lists the proposed variants, the necessary
	quasi-Newton baselines, and the full-model comparisons. The
	abbreviations are used consistently in all numerical tables and figures.
	
	\begin{table}[t]
		\centering
		\caption{Algorithms to be compared in the numerical study.}
		\label{tab:methods}
		\small
		\begin{tabular}{L{0.20\linewidth}L{0.70\linewidth}}
			\toprule
			Abbreviation & Description \\
			\midrule
			\THMLBFGSN &
			Proposed preconditioned ordinary-hyperbolic modified L-BFGS with the negative-curvature reflection
			$\ytilde=-y$. \\
			\THMLBFGSE &
			Proposed preconditioned ordinary-hyperbolic modified L-BFGS with the Euclidean
			nearest-point correction. \\
			\THMLBFGSH &
			Proposed preconditioned ordinary-hyperbolic modified L-BFGS with the
			$B_{\mathcal M,P}^{-1}$-metric nearest-point correction. \\
			SK--HMLBFGS &
			ordinary-hyperbolic L-BFGS that skips weak and nonpositive
			curvature pairs; it is the cautious-update baseline in the spirit of
			globalized L-BFGS~\cite{Mannel2025}. \\
			PD--HMLBFGS &
			ordinary-hyperbolic L-BFGS with Powell damping of the curvature
			pair~\cite{NocedalWright2006,Powell1978}. \\
			RL--HMLBFGS &
			Regularized limited-memory quasi-Newton method based on
			$B_k+\xi_kI$, implemented following the regularization framework
			of~\cite{KanzowSteck2023}. \\
			THSNM &
			The ordinary hyperbolic smoothing Newton method
			of~\cite{ZhaoTang2026}. \\
			HSNM &
			The full ordinary hyperbolic smoothing Newton counterpart
			of THSNM~\cite{ZhaoTang2026}. \\
			TASNM &
			The aggregate smoothing Newton method
			of~\cite{XiaoYu2010}. \\
			FMINIMAX &
			The MATLAB \texttt{fminimax} solver, used only as a software
			reference where its problem interface applies. \\
			\bottomrule
		\end{tabular}
	\end{table}
	
	\subsection{Shared settings and computing environment}
	
	The implementation uses the geometric continuation
	\[
	\tau_k=\tau_0\,\texttt{tauFactor}^{\,k},\qquad
	\tau_0>0,\quad0<\texttt{tauFactor}<1,
	\]
	and the Armijo choices $\alpha=10^{-4}$, $\beta=0.5$, with initial
	trial step one.  Each layer stops when the full smoothed gradient infinity
	norm is at most \texttt{tol}, subject to the explicit \texttt{innerMax}
	and \texttt{outerMax} safety caps.  For a preconditioned run, the common
	parameters are $(m_L,\deltac,\Lc,\chi_{\min},\chi_0,\chi_{\max},\delta,\bar L)$.
	The same values must be used for \THMLBFGSN, \THMLBFGSE, and \THMLBFGSH.
	The reported $\bar L$ must be a valid bound on the region traversed by
	the iterates and line-search trials, as explained in Remark~\ref{rem:Lbar-scope}.
	
	All computations are performed in MATLAB R2020b on a laptop with an AMD
	Ryzen~7 4800H CPU at 2.90\,GHz, 16\,GB RAM, and a 64-bit operating system,
	which is the computing environment reported in~\cite{ZhaoTang2026}.
	CPU time is the average over ten independent runs. For the Newton
	references, MATLAB functions \texttt{chol}, \texttt{rcond}, and
	\texttt{eig} are used exactly for the factorization, condition estimate,
	and smallest-eigenvalue calculation described in that reference; FMINIMAX
	is obtained by calling \texttt{fminimax}.
	
	\subsection{Test suite and reported quantities}
	
	The benchmark suite should first reproduce the finite minimax examples in
	Section~6 of~\cite{ZhaoTang2026}, including smooth and nonconvex component
	functions and problems with a large number of components. Additional
	finite-max quadratic and robust least-squares tests may be added only when
	the component functions, initial point, known reference value (if used),
	and stopping rule are stated in full.
	
	For each method and problem, we report
	\begin{itemize}
		\item function evaluations, gradient evaluations, and wall-clock time;
		\item final objective value, final full-smoothed gradient infinity norm,
		and the final smoothing parameter;
		\item the number of directly stored positive-curvature pairs, screened
		pairs, and correction-eligible negative-curvature pairs;
		\item the number of accepted corrections of each type;
		\item the curvature ratio $|s_k^Ty_k|/(s_k^TB_{\mathcal M,P}s_k)$ when the $H$-metric
		correction is used;
		\item the descent angle
		$-(g_k^\ell)^Td_k^\ell/(\norm{g_k^\ell}\norm{d_k^\ell})$ and the number of memory resets.
	\end{itemize}
	
	The resulting comparisons test three separate questions: whether the
	proposed pair corrections retain information that SK--THMLBFGS discards;
	whether the two minimum-change rules behave differently from direct sign
	reflection; and whether an anisotropic pair correction is competitive with
	the isotropic regularization in RL--THMLBFGS. These are empirical questions;
	no numerical superiority is claimed without the corresponding tables.
	
	% ============================================================
	\section{Conclusions and Future Work}
	\label{sec:conclusion}
	% ============================================================
	
	This final section summarizes the method and identifies directions needed to
	turn the present analysis and comparison protocol into a broader numerical
	and algorithmic study.
	
	This paper develops a preconditioned HMLBFGS framework for finite minimax
	problems with ordinary hyperbolic smoothing.  The hyperbolic-majorization
	matrix $P_k$ is constructed from the current residuals and component
	gradients, and supplies the initial inverse metric
	$H_{k,0}=\chi_kP_k^{-1}$.  The matching direct initialization
	$B_{k,0}=P_k/\chi_k$ is used in the H correction.  Only the curvature
	vector in a stored pair is modified when $s_k^Ty_k<0$.
	
	The three corrections have complementary interpretations. Direct sign
	reflection is the cheapest baseline. The Euclidean nearest-point rule is an
	exact minimum-change correction in the ordinary norm. The
	$B_{\mathcal M,P}^{-1}$-metric
	rule is an exact minimum-change correction in the geometry learned by the
	current quasi-Newton model and is evaluated by the direct compact L-BFGS
	product $\mathcal B_{\mathcal M}s_k$. All three rules produce the same
	positive secant curvature and therefore preserve positive definiteness and
	reliable first-order descent.
	
	The analysis proves positive definiteness, the global quadratic majorizer,
	and fixed-smoothing-layer L-BFGS spectral bounds.  These establish
	fixed-layer first-order convergence and Q-linear objective convergence under
	a PL condition; local strong convexity yields R-linear iterates.  A key
	boundary is explicit: with the current Euclidean pair safeguards, the
	analysis does not assert a $\tau$-uniform L-BFGS energy bound.  The ideal
	geometric continuation with vanishing inner tolerances yields Clarke
	stationarity, while the finite code with a fixed tolerance yields an
	approximate stationary point of its last smoothed layer.  Near a strongly
	convex fixed-smoothing solution, the corrections are eventually inactive
	and the corresponding full-memory BFGS method has the standard
	Q-superlinear conclusion under eventual unit steps.
	Future work has four natural directions: completing the controlled numerical
	comparisons in Section~\ref{sec:numerics}; deriving sharper condition-number
	bounds for the metric projection; choosing the correction mode adaptively
	from the observed pair geometry; and combining the present SPD L-BFGS metric
	with a separate saddle-escape mechanism when strict second-order stationarity
	is desired.
	
	% ============================================================
	% References
	% ============================================================

\end{document}